\documentclass{article}

\usepackage{PRIMEarxiv}

\usepackage[utf8]{inputenc} 
\usepackage[T1]{fontenc}    
\usepackage{hyperref}       
\usepackage{url}            
\usepackage{booktabs}       
\usepackage{amsfonts}       
\usepackage{nicefrac}       
\usepackage{microtype}      
\usepackage{lipsum}
\usepackage{fancyhdr}       
\usepackage{graphicx}       
\graphicspath{{media}}     

\usepackage{graphicx}
\usepackage{algorithm}
\usepackage{subfigure}
\usepackage{xcolor}
\usepackage{mathptmx}      
\RequirePackage{amsthm,amsmath,amsfonts,amssymb, anysize, setspace}
\usepackage{amsthm,amsmath,amsfonts,amssymb}
\usepackage{latexsym}

\theoremstyle{plain}
\newtheorem{theorem}{Theorem}[section]
\newtheorem{corollary}[theorem]{Corollary}
\newtheorem{lemma}[theorem]{Lemma}
\newtheorem{proposition}[theorem]{Proposition}
\theoremstyle{definition}
\newtheorem{definition}{Definition}

\theoremstyle{remark}
\newtheorem{remark}{Remark}

\title{Infinite-server systems with \\
Hawkes arrivals and Hawkes services
\thanks{\textit{\underline{Citation}}: 
\textbf{Dharmaraja Selvamuthu and Paola Tardelli. 
Infinite-server Systems with
Hawkes Arrivals and Hawkes Services. 
To appear on Queueing Systems.}} 
}

\author{
  Dharmaraja Selvamuthu \\
  Department of Mathematics\\
  Indian Institute of Technology Delhi\\
  Hauz Khas, New Delhi-110016, India \\
  \texttt{dharmar@maths.iitd.ac.in} \\
   \And
  Paola Tardelli  \\
  Department of Industrial and \\ Information Engineering and Economics \\
  University of L'Aquila \\
  Piazzale E. Pontieri, Roio, 67100, Italy \\
  \texttt{paola.tardelli@univaq.it} \\
}

\begin{document}
\maketitle

\begin{abstract}
This paper is devoted to the study of the number of customers in  infinite-server systems driven by Hawkes processes.
In these systems, the self-exciting arrival process is assumed to be represented by a Hawkes process 
and the self-exciting service process by a state-dependent Hawkes process (sdHawkes process). 
Under some suitable conditions, for the $Hawkes/sdHawkes/\infty$ system, 
the Markov property of the system is derived. 
The joint time-dependent distribution of the number of customers in the system, 
the arrival intensity and the server intensity is characterised by a system of differential equations.
Then, the time-dependent results are also deduced for the   
 $M/sdHawkes/\infty$ system.
\end{abstract}


\section{Introduction}
\label{intro}
In fields like probability, operations research, management science, and industrial engineering, there are phenomena that are well described by infinite-server systems. Many authors, in literature,
assume that in most queueing systems 
the arrival process of customers is a Poisson process and 
the service times have an exponential distribution. 
This implies that the future evolution of the system only depends on the present situation. 

However, in real-world problems, 
the future evolution of the arrival process and/or the service process is 
often influenced by the past events. 
This is essentially observed, during periods of suffering, for instance earthquakes, crime and riots 
and social media trend (Daw and Pender 
\cite{DP}, 
and 
Rizoiu et al. 
\cite{RXSCYH}) and in the financial market, incorporating some kind of contagion effect
(A{\"\i}t-Sahalia et al. 
\cite{ait2015}). 
Recently, the same phenomenon is observed to describe the temporal growth and migration of COVID-19 
(see  Chiang et al. 
\cite{CLM2021}, Escobar 
\cite{Escobar2020},  Garetto et al. 
\cite{GLT2021}).
The work by Eick et al. 
\cite{EMW}, and recent works by Daw and Pender 
\cite{DP} and Koops et al. 
\cite{KSBM}, 
motivated us to consider the service process as a state-dependent Hawkes (sdHawkes) process.

This 
article studies an infinite-server system wherein the dynamics of the arrival process and the service process are governed by Hawkes processes. 
The self-exciting property of Hawkes process is well known for its efficiency in reproducing jump clusters due to the dependence of the conditional event arrival rate, or intensity, on the counting process itself. 
It can be considered that 
this type of infinite-server system has the potential to represent, for instance, the evolution of the number of people visiting a shopping center or the number of clients visiting a website. 
In both these cases, the intensity function for the arrival process cannot be considered as a constant and,  the rate of new arrival events increases with the occurrence of each event. 
Similarly, it is  observed that the intensity function for the service process may be state-dependent and every new service completion excites further service process.

Errais et al. 
\cite{EGG2010} derived
the Markov property for the two-dimensional process consisting of 
a Hawkes process and its intensity having  exponential kernel. 
We generalize their approach investigating 
the Markov property for the three-dimensional process consisting of the number of customers in the system, 
the arrival intensity and the service intensity. 

To achieve this, we have a problem that must be solved. On a filtered probability space,  $(\Omega, \mathcal{F},\{{\mathcal F}_t\}_{t \geq 0},P)$, we consider the filtrations generated by the arrival process, 
$\{{\mathcal F}^M_t\}_{t \geq 0}$, the filtrations generated by the service process 
$\{{\mathcal F}^S_t\}_{t \geq 0}$, and the filtrations generated by  the number of the customers in the system $\{{\mathcal F}^N_t\}_{t \geq 0}$. 
Then, we note that ${\mathcal F}^N_t \subseteq {\mathcal F}^M_t \vee 
{\mathcal F}^S_t$, which implies that every ${\mathcal F}^M_t \vee 
{\mathcal F}^S_t$-martingale is a 
${\mathcal F}^N_t$-martingale. 
But, the vice versa is not always true, namely,  a 
${\mathcal F}^N_t$-martingale is not always a ${\mathcal F}^M_t \vee 
{\mathcal F}^S_t$-martingale, in general. This is the classical problem of the enlargement of filtrations. 
An assumption in this framework is the 
Immersion property, 
which allows us to get that any 
${\mathcal F}^N_t$-martingale is also a  ${\mathcal F}^M_t \vee 
{\mathcal F}^S_t$-martingale (see 
 Bielecki et al. 
\cite{BJR2006}, Jeanblanc et al.
\cite{JYC2009}, 
and Tardelli 
\cite{T2017}). 
Consequently, following the existing literature, the Immersion property is assumed in the present paper.


Generalizing some results of Daw and Pender \cite{DP} and Koops et al. \cite{KSBM}, we derive the transient or time-dependent behaviour of the infinite-server system with Hawkes arrivals and sdHawkes services. 
This is the main contribution of the present paper.

The paper is organized as follows. In Section 2, we introduce Hawkes processes and their properties. 
In Section 3, we describe an infinite-server system with Hawkes arrivals and sdHawkes services.
We discuss the moments of the arrival process and its intensity.  We derive the Markov property of the process describing the number of customers in the system using
the Immersion property. 
In Section 4, we obtain
the joint transient or time-dependent distribution of the system size, the arrival and the service intensity processes for 
the $Hawkes/sdHawkes/\infty$ system. 
Afterwards, we deduce the above time-dependent results  for a  
$M/sdHawkes/\infty$
system. 
Finally, in Section 5, we present concluding remarks.

\section{Hawkes Processes}

Hawkes processes constitute a particular class of multivariate point processes having numerous applications throughout science and engineering. 
These are processes characterised by a particular form of stochastic intensity vector, 
that is, the intensity depends on the sample path of the point process. 

Let $(\Omega, \mathcal{F},\{{\mathcal F}_t\}_{t \geq 0},P)$ be a filtered probability space, 
where $\{{\mathcal F}_t\}_{t \geq 0}$ is a given filtration satisfying the usual conditions. On this  space, a counting process $M = \{M_t, t \geq 0 \}$ is defined, and $\{{\mathcal F}^M_t\}_{t \geq 0}$, where 
$\mathcal{F}^M_t = \sigma\{M_u, u \leq t\}$, is  its associated filtration, 
and stands for the information available up to 
time $t$. 
\bigskip

\begin{definition}
The conditional law of $M$ is defined, for 
$\Delta t \rightarrow 0$,
as
$$P[M_{t+\Delta t} - M_t = m | {\cal F}^M_t] = 
\left\{ 
\begin{array}l 
\lambda_t \Delta t + o(\Delta t) \\
o(\Delta t) \\
1 - \lambda_t \Delta t + o(\Delta t) 
\end{array} 
\quad 
\begin{array}l 
m = 1 \\
m > 1\\
m = 0.
\end{array} 
\right.$$
For a Hawkes process, the intensity $\lambda_t$ is a function of the past jumps of the process itself, and, in general, assumes the representation
\begin{equation}
\label{generallambda}
\lambda_t = \lambda_0 +\int_0^t \varphi(t-u) dM_u.
\end{equation}
The function $\varphi(\cdot)$, called {\it excitation function}, is such that $\varphi(t) \geq 0$, for $t \geq 0$. It represents the size of the impact of jumps and belongs to the space of $L^1$-integrable functions. 
\end{definition}
\bigskip

When $\varphi(t)=0$, $M$ is a counting process with constant intensity.
This means that a Poisson process is obtained as a special case of a Hawkes process.

Following the existing literature
(see Daw and Pender 
\cite{DP}),  we restrict our attention to a function 
$\varphi(\cdot)$ defined by an  exponential decay kernel.
To this end, 
let the arrival intensity be governed by the dynamics
\begin{equation}
\label{dynamicsoflambda}
d\lambda_t = r (\lambda^* - \lambda_t) dt + B_t dM_t,
\end{equation}
where 
$\lambda^*$ represents an underlying stationary arrival rate, called baseline intensity.
The constant $r > 0$ describes the decay of the intensity as time passes after an arrival,  
and $B_t$, for each $t \geq 0$, is a positive random variable representing  the size of the jump in the intensity upon an arrival.
The solution to Equation (\ref{dynamicsoflambda}), 
given $\lambda_0$, which is the initial value of $\lambda_t$, is obtained as
\begin{equation}
\label{lambda}
\lambda_t = \lambda^* + e^{-rt} (\lambda_0 - \lambda^*) + \int_0^t B_s  e^{-r (t - s)} dM_s.
\end{equation}
Taking $\{t_i\}_{i >0}$ as the sequence of the jump times of $M$, 
and if the self-excting term is such that $\varphi(t - t_i) \equiv B_{t_i} e^{-r(t-t_i)} $, 
then Equation (\ref{generallambda}) coincides with Equation (\ref{lambda}).

\begin{remark}
Note that, the Hawkes process $M$ 
itself does not have the Markov property. However, assuming that the 
excitation function $\varphi(\cdot)$ has an exponential decay kernel, it is possible to prove that the bi-dimensional process $(\lambda, M)$ is a Markov process (Errais et al. \cite{EGG2010}). 
Furthermore,  the explosion is avoided by ensuring the condition given by $\mathbb E[B_t] < r$, (Daw and Pender  \cite{DP}). 

\end{remark}

\section{The Model:  $Hawkes/sdHawkes/\infty$ System}

The arrival process  of the system is the  Hawkes process $M$ 
acting as an input process to an infinite-server system, having arrival intensity 
given by Equation (\ref{lambda}), and $\{t_i\}_{i \geq 0}$ as  the sequence of the  arrival times. 

For the service requirement, we consider $S = \{S_t, t \geq 0\}$, another Hawkes process having serving intensity $\mu_t$, 
with initial intensity $\mu_0 > 0$, baseline intensity $\mu^*$ and, 
exponential excitation function. 
Taking $N_t$ as the number of customers in the system at $t$, 
for a constant $s > 0$, and for $t > 0$, given by
\begin{equation}
\label{mu}
\mu_t = N_t \left( \mu^* + e^{-st} (\mu_0 - \mu^*) + \int_0^t C_u  e^{-s (t - u)} dS_u \right).
\end{equation}
For $i>0$,  
$\tau_i$ denotes the time epoch 
of $i$th customer departure 
after the service completion, 
and $C_{\tau_i}$ is a positive random variable representing  the size of that jump.  
The 
number of customers in the system whose service is completed on, or before, time $t$ is given by $S_t$ 
and $N_t = M_t - S_t$.

This form of the serving intensity $\mu_t$ allows us to take into account the crowding of the system, which is the number of customers in the system at time $t$,
and  about the experience gained by the server process, which is given by the Hawkes structure. 
This is called as a {\it state-dependent Hawkes process}, sdHawkes process 
(Li and Cui 
\cite{LC2020} and Morario-Patrichi and Pakkanen 
\cite{MPP2018}). 

To the best of our knowledge, this is the first time in literature that Hawkes processes are introduced to model 
the experience in arrivals and services of an infinite-server system.
Furthermore, we note that, at time $t$, the sdHawkes process   for service will start as new, whenever the number of customers in the system become one, i.e., $N_t =1$. 
Therefore, we call this queueing model as 
$Hawkes/sdHawkes/\infty$ in Kendall's notation. 

Since the service process is state-dependent, 
for the stability of the infinite-server system (Daw and Pender 
\cite{DP}), only the stability for the arrival process is needed, 
which means that 
$\mathbb E[B_t] <  r$, for all $t$.

\subsection{Moments of $M_t$ and $\lambda_t$}

This subsection is devoted to 
derive the results regarding the moments of the process $M$ and its intensity $\lambda$.  These results can be proved  
taking into account the results of Dassios and Zhao 
\cite{DZ}
and generalizing those obtained in Daw and Pender 
\cite{DP}, Section 2.
\bigskip

\begin{proposition}
Given a Hawkes process $(M_t, \lambda_t)$,  with dynamics given by Equation (\ref{lambda}) 
and $\mathbb E[B_t] <  r$, for all $t \geq 0$,
\begin{eqnarray}
\label{Elambdat}
\mathbb E[\lambda_t] 
&=& \frac{r \lambda^*}{r - \mathbb E[B_t]} 
+ \left(\lambda_0 -  \frac{r \lambda^*}{r - \mathbb E[B_t]}  \right) e^{ - t \left(r - \mathbb E[B_t]\right) },
\\
\label{Varlambdat}
Var[\lambda_t] 
&=& \frac{(\mathbb E[B_t])^2}{r - \mathbb E[B_t]} 
\left[\left( \frac{r \lambda^*}{2 \left(r - \mathbb E[B_t] \right)}  - \lambda_0\right) e^{ - 2 t \left(r - \mathbb E[B_t]\right) }
- \left(  \frac{r \lambda^*}{r - \mathbb E[B_t]} - \lambda_0  \right) e^{ - t \left(r - \mathbb E[B_t]\right) }
+ \frac{r \lambda^*}{2 \left(r - \mathbb E[B_t] \right)} \right],
\\
\label{EMt}
\mathbb E[M_t] 
&=& \frac{1}{r - \mathbb E[B_t]} \left[ r \lambda^* t 
+  \left(\lambda_0 -  \frac{r \lambda^*}{r - \mathbb E[B_t]}  \right) 
\left(1 - e^{ - t \left(r - \mathbb E[B_t]\right) } \right) \right],
\\
\label{VarMt}
Var[M_t] 
&=&  \frac{1}{(r - \mathbb E[B_t])^3} \left\{ r^3 \lambda^* t  
+ r^2 \left( \lambda_0 - \frac{r \lambda^*}{2 \left(r - \mathbb E[B_t] \right)} \right) 
\left(1 - e^{ - 2t \left(r - \mathbb E[B_t]\right) } \right) 
\right.
\\
&&  
- 2r \mathbb E[B_t] (r - \mathbb E[B_t]) \left( \lambda_0 - \frac{r \lambda^*}{r - \mathbb E[B_t]} \right) 
t e^{ - t \left(r - \mathbb E[B_t]\right) } 
\nonumber\\
&& 
\left.
+ 
\left( \left(r^2 - (\mathbb E[B_t] )^2 \right) \lambda_0 
-  r \lambda^* \frac{r^2 - (\mathbb E[B_t] )^2+ 2r \mathbb E[B_t]}{\left(r - \mathbb E[B_t] \right)} \right)  
\left(1 - e^{ - t \left(r - \mathbb E[B_t]\right) } \right) 
\right\}.
\nonumber
\end{eqnarray}
\end{proposition}
\bigskip

\begin{remark}
In general, for the moments of $M_t$ and for the moments of $\lambda_t$ of order $n$,
we have to recursively solve the following system of differential equations 
\begin{eqnarray}
\frac{d \mathbb E[M^n_t]}{dt} &=& \sum_{i=0}^{n-1}  {n \choose i} \mathbb E[\lambda_t M^i_t],
\nonumber\\
\frac{d \mathbb E[\lambda^n_t]}{dt} &=& n r \lambda^* \mathbb E[\lambda^{n-1}_t] - n r \mathbb E[\lambda^n_t] 
+  \sum_{i=0}^{n-1}  {n \choose i} (\mathbb E[B_t] )^{n-i} \mathbb E[\lambda^{i+1}_t],
\nonumber\\
\frac{d \mathbb E[\lambda^n_t M^k_t]}{dt} &=& n r \lambda^* \mathbb E[\lambda^{n-1}_t M^k_t] 
- n r \mathbb E[\lambda^n_t M^k_t] 
+  \sum_{(i,j) \in \cal S}  {n \choose i}  {k \choose j} (\mathbb E[B_t] )^{n-i} \mathbb E[\lambda^{i+1}_t M^j_t],
\nonumber
\end{eqnarray}
for positive integer values of  $n$ and $k$ and with 
${\cal S} := \{(i, j): i = 0, \ldots, n, j = 0, \ldots, k, (i,j) \neq (n,k)\}$. 
\end{remark}
\bigskip

\begin{corollary}
\label{cor1}
If $\mathbb E[B_t] = r$, 
the differential equations given in Daw and Pender 
\cite{DP} imply that
$$\mathbb E[\lambda_t] = \lambda_0 + r \lambda^* t, \qquad 
Var[\lambda_t] = r^2 \left(  \lambda_0 t + \frac{r \lambda^*}{2} t^2 \right),
\qquad 
\mathbb E[M_t]  = \lambda_0 t + \frac{r \lambda^*}{2} t^2,$$
$$Var[M_t]= \lambda_0 t + r \left( \lambda_0 + \frac{\lambda^*}{2} \right) t^2 
+ \frac{r^2}{3} \left( \lambda_0 + \lambda^*\right) t^3 + \frac{r^2 \lambda^*}{6} \left( r + 3 \lambda^*\right) t^4.$$
\end{corollary}
\bigskip

Note that by taking the limit as $\mathbb E[B_t] \rightarrow r$ in 
Equations (\ref{Elambdat}), (\ref{Varlambdat}),  (\ref{EMt}) and  (\ref{VarMt}), the same results of Corollary \ref{cor1} are achieved.
\bigskip

\begin{proposition}
Assuming that $\mathbb E[B_t] <  r$, and taking
$\displaystyle \lim_{t \rightarrow \infty} \mathbb E[B_t] = \hat B_1$ and 
$\displaystyle \lim_{t \rightarrow \infty} (\mathbb E[B_t] )^2 = \hat B_2$, 
we have 
$$\lim_{t \rightarrow \infty} \mathbb E [\lambda_t] = \frac{r \lambda^*}{r - \hat B_1}  = \lambda_\infty,
\qquad 
\lim_{t \rightarrow \infty}
Var[\lambda_t] 
= \frac{r \lambda^* \hat B_2}
{2 \left(r - \hat B_1 \right)^2},
$$
and  
$\mathbb E[M_t] \rightarrow \infty$ as $t \rightarrow \infty$.
\end{proposition}
\bigskip


\subsection{Markov Property of the $Hawkes/sdHawkes/\infty$  System} 

Recall that, if the excitation function of a Hawkes process is exponential, then the process 
jointly 
with its intensity is a Markov process. 
In order to characterize ${\cal L}\left(N_t, \lambda_t, \mu_t \right)$, the law of $N_t$, the Hawkes arrival intensity $\lambda_t$ and the state dependent Hawkes server intensity $\mu_t$, we have to prove  the Markov property of the $Hawkes/sdHawkes/\infty$  system. 
To this end, we need some preliminaries.
Let the processes $M$, $S$ and $N$ be  defined on the same filtered probability space 
$(\Omega, \mathcal{F},\{{\mathcal F}_t\}_{t \geq 0},P)$. 
Given the sub $\sigma$-algebras of $\mathcal{F}_t$
$$\mathcal{F}^M_t := \sigma\{M_u, u \leq t\}, \qquad \mathcal{F}^S_t := \sigma\{S_u, u \leq t\}, \qquad \mbox{and} \qquad 
\mathcal{F}^{M, S}_t := \mathcal{F}^M_t \vee \mathcal{F}^S_t,$$
we observe that ${\cal F}^{M, S}_t$ contains, also,  all the informations related to the process 
$N = \{N_t, t \geq 0\}$ until time $t$.

Note that, 
$\mathcal{F}^M_t \subseteq 
\mathcal{F}^{M,S}_t$ and 
that 
$\mathcal{F}^S_t \subseteq 
\mathcal{F}^{M,S}_t$. 
Consequently, every 
$\mathcal{F}^{M,S}_t$-martingale is a $\mathcal{F}^M_t$-martingale and, at the same time, every 
$\mathcal{F}^{M,S}_t$-martingale is a $\mathcal{F}^S_t$-martingale. 
But, in general, 
it is not true that a  $\mathcal{F}^M_t$-martingale is a 
$\mathcal{F}^{M,S}_t$-martingale 
and that a 
$\mathcal{F}^S_t$-martingale is a 
$\mathcal{F}^{M,S}_t$-martingale. 
This is a classical topic in  this context,  the so-called problem of the enlargement of filtrations. 
An exciting example is Azema's martingale, see Subsection 4.3.8 in Jeanblanc et al.  
\cite{JYC2009}. 
Hence, to overcome this difficulty,
we assume 
a property for the class of martingales, that is the Immersion property as given below.
\bigskip

 \begin{definition}
 When  the filtration 
 $\mathcal{F}^M := \{{\mathcal F}^M_t\}_{t \geq 0}$ 
 is immersed in the filtration
 $\mathcal{F}^{M, S} := \{{\mathcal F}^{M, S}_t\}_{t \geq 0}$, 
 it means that any
$\mathcal{F}^M_t$-martingale is a 
$\mathcal{F}^{M,S}_t$-martingale. 
 In this case, we say that $\mathcal{F}^M$ 
satisfies the \underline{Immersion Property} with respect to the filtration $\mathcal{F}^{M, S}$.
Also, let  the filtration 
 $\mathcal{F}^S := \{{\mathcal F}^S_t\}_{t \geq 0}$ 
satisfy the Immersion Property  with respect to  
$\mathcal{F}^{M, S}$. 
\end{definition}
\bigskip

\begin{remark}
Given the arrival process $M$ and the service process $S$, we have that 
the sub $\sigma$-algebras  $\mathcal{F}^M_t$ and $\mathcal{F}^S_t$ generated by these processes, respectively, are such that $\mathcal{F}^M_t \subseteq 
\mathcal{F}^{M,S}_t$ and 
$\mathcal{F}^S_t \subseteq 
\mathcal{F}^{M, S}_t$.

By Equation (\ref{mu}), we are able to deduce that the service process $S$ can be represented in terms of a function of $M$ driven by another Hawkes process  $R$, which is independent of $M$. 
Since the processes $M$ and $R$ are independent, any $\mathcal{F}^M_t$-martingale is a   $\mathcal{F}^M_t \vee \mathcal{F}^R_t$-martingale  and any 
$\mathcal{F}^S_t$-martingale is a  $\mathcal{F}^M_t \vee \mathcal{F}^R_t$-martingale. 
Taking into account that $\mathcal{F}^{M, S}_t 
\subseteq 
\mathcal{F}^M_t \vee \mathcal{F}^R_t$,
we get that
$$\mathcal{F}^M_t \subseteq 
\mathcal{F}^{M,S}_t 
\subseteq 
\mathcal{F}^M_t \vee \mathcal{F}^R_t
\qquad \mbox{and} \qquad 
\mathcal{F}^S_t \subseteq 
\mathcal{F}^{M, S}_t
\subseteq 
\mathcal{F}^M_t \vee \mathcal{F}^R_t,$$
which implies that any $\mathcal{F}^M_t \vee \mathcal{F}^R_t$-martingale is a 
$\mathcal{F}^{M, S}_t$-martingales.
As a conclusion, 
all the $\mathcal{F}^M_t$-martingales  and all the $\mathcal{F}^S_t$-martingales are $\mathcal{F}^{M, S}_t$-martingales. 
Greater details on this topic can be found in Aksamit and Jeanblanc \cite{AJ}, Jeanblanc et al. \cite{JYC2009} 
and, more recently, in Calzolari and Torti \cite{CT}. 
\end{remark}
\bigskip

\begin{lemma}
\label{lambdamut+dt}
By Equation (\ref{lambda}), for any 
positive value of 
$\Delta t$,
we get that
\begin{eqnarray}
\lambda_{t+ \Delta t} - \lambda_t  
&=& (e^{-r \Delta t} - 1) e^{- r t} (\lambda_0 - \lambda^*) 
+ (e^{-r \Delta t} - 1) \int_0^{t} B_s  e^{-r (t - s)} dM_s 
+ \int_t^{t+\Delta t} B_s  e^{-r (t +\Delta t - s)} dM_s
\nonumber\\
&=& (e^{-r \Delta t} - 1) (\lambda_t - \lambda^*) + \int_t^{t+ \Delta t} B_s  e^{-r (t + \Delta t - s)} dM_s.
\nonumber
\end{eqnarray}
By Equation (\ref{mu}), if  $N_t$ remains constant between $t$ and $t + \Delta t$, we have that
\begin{eqnarray}
\mu_{t+ \Delta t} - \mu_t 
&=& (e^{- s \Delta t}-1) N_t \left(  
e^{- s t} (\mu_0 - \mu^*) +  \int_0^t C_u  e^{-s (t - u)} dS_u \right) + N_t\int_t^{t+\Delta t} C_u  e^{-s (t + \Delta t - u)} dS_u 
\nonumber\\
&=& (e^{- s \Delta t}-1) ( \mu_t - \mu^* N_t ) + N_t \int_t^{t+\Delta t} C_u  e^{-s (t + \Delta t - u)} dS_u .
\nonumber
\end{eqnarray}
Moreover, 
for a negligible value of 
$\Delta t$ and 
between $t$ and $t+\Delta t$, if there are no new arrivals  and 
there are no service completions, 
we get that 
$$\lambda_{t+ \Delta t} - \lambda_t \approx r \Delta t (\lambda^* - \lambda_t),
\qquad \qquad 
\mu_{t+ \Delta t} - \mu_t \approx s \Delta t (\mu^* N_t - \mu_t).$$
\end{lemma}
\bigskip

Recalling that the self-exciting terms of the intensities $\lambda_t$ and $\mu_t$ 
are defined by exponential decay kernels, 
even though the processes $M$ and $S$ does not have the Markov property themselves, 
we are able to show that  $(M, S, \lambda, \mu)$ is a Markov process 
(Errais et al. 
\cite{EGG2010}). 
To get the Markov property of $(M, S, \lambda, \mu)$, 
we derive its Dynkin formula, 
in the next Proposition, 
taking into account that this formula is a direct consequence of the strong Markov property and, hence, it
builds a bridge
between differential equations and Markov processes, 
(see for instance, \O ksendal 
\cite{O2003},  Section 7.4).
\bigskip

\begin{proposition}
\label{generators-Markovianity}
Let $\mathcal{A}$ be an operator acting on a suitable function $f$, with continuous partial derivatives with respect to $\lambda$ and $\mu$, such that
\begin{equation}
\label{generator}
\mathcal{A}f(M, S, \lambda, \mu) 
:= r (\lambda^* \hspace{-.02 in} - \lambda) \frac{\partial f(M, S, \lambda, \mu)}{\partial \lambda} 
+ s (\mu^* - \mu) \frac{\partial f(M, S, \lambda, \mu)}{\partial \mu}
+ \lambda \mathcal{A}^\lambda f(M, S, \lambda, \mu) + \mu  \mathcal{A}^\mu f(M, S, \lambda, \mu),
\end{equation}
where
$$\mathcal{A}^\lambda f(M, S, \lambda, \mu)  = \int_0^\infty \left[f\left(M+1, S, \lambda + x, \mu \right) - f(M, S, \lambda, \mu) \right] dP[B \leq  x],$$
$$\mathcal{A}^\mu f(M, S, \lambda, \mu)  
=  \int_0^\infty \left[f\left(M, S+ 1, \lambda, \mu + y \right) - f(M, S, \lambda, \mu) \right]  dP[C \leq y],$$
and  $B$ and $C$ are random variables such that 
$P(B > 0) = 1$ and $P(C > 0) = 1$.  If, for $t \leq T$,
$$\mathbb E\left[ \int_0^t \left| \mathcal{A} f(M_u, S_u, \lambda_u, \mu_u)
- r (\lambda^* - \lambda_u) \frac{\partial f(M_u, S_u, \lambda_u, \mu_u)}{\partial \lambda} 
- s (\mu^* - \mu_u) \frac{\partial f(M_u, S_u, \lambda_u, \mu_u)}{\partial \mu} 
\right| \ du \right] < \infty,
$$
the following Dynkin formula holds
\begin{equation}\label{Dynkin}
\mathbb E \left[ f (M_T, S_T, \lambda_T, \mu_T) \Big| \mathcal{F}^{M, S}_t \right]
= f (M_t, S_t, \lambda_t, \mu_t)  
+ 
\mathbb E \left[ \int_t^T \mathcal{A} f (M_u, S_u, \lambda_u, \mu_u) \ du \Big| \mathcal{F}^{M, S}_t  \right].
\end{equation}
\end{proposition}
\bigskip

\proof 
For the sake of completeness, we prove this result along similar line as in Errais et al.  
\cite{EGG2010}. 
For a fixed time $t$, $(M_t, S_t, \lambda_t, \mu_t)$ has right-continuous paths of finite variation.
Hence,
\begin{eqnarray}\label{cont+jumps}
f (M_t, S_t, \lambda_t, \mu_t) - f (M_0, S_0, \lambda_0, \mu_0) 
&=&
\int_0^t  r (\lambda^* - \lambda_u) \frac{\partial f(M_u, S_u, \lambda_u, \mu_u)}{\partial \lambda}  \ du 
\nonumber\\
&&
+ \int_0^t  s (\mu^* - \mu_u) \frac{\partial f(M_u, S_u, \lambda_u, \mu_u)}{\partial \mu}\ du
\\
&&
+ \sum_{0 < u \leq t} \left[ f(M_u, S_u, \lambda_u, \mu_u) - f(M_{u-}, S_{u-}, \lambda_{u-}, \mu_{u-}) \right].
\nonumber
\end{eqnarray}
Note that 
$m^\lambda_t = M_t - \int_0^t \lambda_v dv$ 
is a $\mathcal{F}^M_t$-martingale, and by the Immersion Property, $m^\lambda_t$ is also a $\mathcal{F}^{M, S}_t$-martingale, and, moreover, 
$m^\mu_t = S_t - \int_0^t \mu_v dv$ is a $\mathcal{F}^{M, S}_t$-martingale.
Hence, we are able to rewrite the summation in Equation (\ref{cont+jumps}) as
$$\int_0^t \mathcal{A}^\lambda f(M_v, S_v, \lambda_v, \mu_v)  \left( dm^\lambda_v + \lambda_v  dv \right)
+ \int_0^t \mathcal{A}^\mu f(M_v, S_v, \lambda_v, \mu_v)  \left( dm^\mu_v + \mu_v  dv \right).$$
The integrability condition on the predictable integrand guarantees that 
$$\int_0^t \mathcal{A}^\lambda f(M_v, S_v, \lambda_v, \mu_v)  \ dm^\lambda_v 
+ \int_0^t \mathcal{A}^\mu f(M_v, S_v, \lambda_v, \mu_v)  \ dm^\mu_v$$
is a martingale (Theorem 8, Chapter II of Bremaud 
\cite{Bremaud1981}).
As a conclusion, we get that
$f(M, S, \lambda, \mu)$ is a semi-martingale
having a unique decomposition given by a sum of a predictable process with finite variation and a martingale, and we are able to write that 
$$f (M_t, S_t, \lambda_t, \mu_t)) - f (M_0, S_0, \lambda_0, \mu_0)) - \int_0^t \mathcal{A} f (M_v, S_v, \lambda_v, \mu_v) dv$$
$$= \int_0^t \mathcal{A}^\lambda f(M_v, S_v, \lambda_v, \mu_v)  dm^\lambda_v + \int_0^t \mathcal{A}^\mu f(M_v, S_v, \lambda_v, \mu_v)  dm^\mu_v.$$
Since the right hand side is a martingale, so is the process defined by the left hand side, which results to the formula given by Equation (\ref{Dynkin}). \qed
\bigskip

Under all the assumptions made so far, the infinitesimal generator of the process $(M, S, \lambda, \mu)$  acting on a function $f(M, S, \lambda, \mu)$ 
is given by Equation  (\ref{generator}).
\bigskip

\begin{remark}\label{Nlambdamu markov}
Since $(M, S, \lambda, \mu)$ is a Markov process, 
and taking into account that the process 
$N = M - S$, 
we are able to deduce that $(N, \lambda, \mu)$ is also a Markov process.
\end{remark}
\bigskip

\section{Characterisation of the Law of the Infinite-server Systems}

The joint transient distribution of $\left(N, \lambda, \mu \right)$ is uniquely defined by the transformation
\begin{equation}
\label{Def-zeta}
\zeta(t, z, u, v) = \mathbb E \left[ z^{N_t} e^{- u \lambda_t} e^{- v \mu_t} \right],
\end{equation}
where $t \geq 0$, $0 \leq z \leq 1$, $u \geq 0$ and $v \geq 0$.
In this section, we characterize $\zeta$ in terms of the solution of a system of ordinary differential equations (ODEs).
\bigskip

\begin{theorem}\label{mainTheorem}
Let the arrival process be a Hawkes process and let the service process be a sdHawkes process. 
Given the random variables $B$ and $C$ as defined in Proposition \ref{generators-Markovianity}, 
let $\beta(u) := \mathbb E [e^{- u B}]$ and let $\gamma(v) := \mathbb E [e^{- v C}]$.\\
If $(N_t, \lambda_t, \mu_t)|_{t=0} = (0, \lambda_0, 0)$, 
then the couple $\left(u(\cdot), v(\cdot) \right)$  solves the system of ODEs
\begin{equation}  
\label{systemuv}
\left\{ 
\begin{array}l 
\displaystyle u'(t) + r u(t) -1 +  z \beta \left(u(t)  \right) e^{   - s \mu^* \int_0^t v(x) \ dx } = 0,
\\
\displaystyle v'(t) + s v(t) - 1 + \frac{\gamma \left(v(t)\right)}{z} e^{ s \mu^* \int_0^t v(x) \ dx } = 0, 
\end{array}
\right.
\end{equation}
with boundary conditions $u(0) = u$ and $v(0) = v$. 
Furthermore, 
\begin{equation}  
\label{EquationZeta}
\zeta(t, z, u, v) = e^{- u(0) \lambda_0 } \exp{\left\{ - \lambda^* r \int_0^t u(x) \ dx \right\}},
\end{equation}
where $\zeta$ depends on $v$ through the coupling with $u$ given by system of ODEs (\ref{systemuv}).
\end{theorem}
\bigskip

In order to prove Theorem \ref{mainTheorem}, we make use of  Proposition \ref{firstpart} and Proposition \ref{secondpart} 
given below, whose proofs are in Appendix.
\bigskip

\begin{proposition}\label{firstpart}
The joint distribution of $\left(N, \lambda, \mu \right)$, for $t > 0$, 
$$F(t, k, \lambda, \mu) = P\left[N_t = k, \lambda_t \leq \lambda, \mu_t \leq \mu \right]$$
is such that
\begin{eqnarray}
\label{LHSFderivated=RHSFderivated}
&&\frac{\partial^3 F(t, k, \lambda, \mu)}{\partial \mu \partial \lambda \partial t} 
+ (\lambda + \mu) \frac{\partial^2 F(t, k, \lambda, \mu)}{\partial \mu \partial \lambda} 
\\
&=& 
\int_0^\lambda x  \frac{\partial^2 F(t, k-1, x, \mu)}{\partial \mu \partial \lambda}   \ dP[B \leq \lambda - x]
-  \frac{\partial}{\partial \lambda} 
\left[ r \left( \lambda^* - \lambda \right) 
\frac{\partial^2 F(t, k, \lambda, \mu)}{\partial \mu \partial \lambda} \right]
\nonumber\\
&&
+ \int_0^\mu y \frac{\partial^2 F(t, k+1, \lambda, y)}{\partial \lambda \partial \mu} \ dP[C \leq \mu - y]
- \frac{\partial}{\partial \mu} 
\left[ s \left( k \mu^* - \mu \right) 
\frac{\partial^2 F(t, k, \lambda, \mu)}{\partial \mu \partial \lambda} \right] .
\nonumber
\end{eqnarray}
\end{proposition}
\bigskip

\begin{proposition}\label{secondpart}
Taking
$$\xi(t, k, u, v) = \int_0^\infty  \int_0^\infty e^{- u \lambda} e^{- v \mu} \ \frac{\partial^2 F(t, k, \lambda, \mu)}{\partial \lambda \partial \mu} \ d\lambda \ d\mu,$$ 
we have that
\begin{eqnarray}
\label{Eq-xi}
&&
\frac{\partial \xi(t, k, u, v)}{\partial t} 
+ (u r - 1) \frac{\partial \xi(t, k, u, v)}{\partial u} + \beta(u) \frac{\partial \xi(t, k-1, u, v) }{\partial u} 
\\
&& 
+ (v s - 1)  \frac{\partial \xi(t, k, u, v)}{\partial v}
+ \gamma(v) \frac{\partial \xi(t, k+1, u, v)}{\partial v} 
=  - ( u r \lambda^* + v s k \mu^* ) \xi(t, k, u, v). 
\nonumber
\end{eqnarray}

\end{proposition}
\bigskip

\proof {\it of Theorem \ref{mainTheorem}}

Recalling the definition of $\zeta$ given in Equation  (\ref{Def-zeta}), which in turn implies that
$$\zeta(t, z, u, v) = \sum_{k=0}^\infty z^k \xi(t, k, u, v),$$
and taking into account that $F(t, -1, \lambda, \mu) = 0$, note that
\begin{eqnarray}
\sum_{k=0}^\infty z^k  \frac{\partial \xi(t, k-1, u, v)}{\partial u}  &=& 
\frac{\partial \xi(t, -1, u, v)}{\partial u} + z \sum_{k=1}^\infty z^{k-1} \frac{\partial \xi(t, k-1, u, v)}{\partial u} 
\nonumber\\
& = & z \frac{\partial}{\partial u}   \sum_{k=0}^\infty z^k \xi(t, k, u, v) = z \frac{\partial \zeta(t, z, u, v)}{\partial u}.
\nonumber
\end{eqnarray}
Moreover, 
\begin{eqnarray}
\sum_{k=0}^\infty z^k  \frac{\partial \xi(t, k+1, u, v)}{\partial v} 
&=& \frac{1}{z} \frac{\partial}{\partial v}  \sum_{k=0}^\infty z^{k+1}  \xi(t, k+1, u, v)
\nonumber\\
& = &  \frac{1}{z} \frac{\partial}{\partial v} \int \left[ \sum_{k=0}^\infty k z^{k-1} \xi(t, k, u, v) \right] \ dz
= \frac{1}{z} \frac{\partial  \zeta(t, z, u, v)}{\partial v},
\nonumber
\end{eqnarray}
and
$$\sum_{k=0}^\infty k z^k \xi(t, k, u, v) 
= z \frac{\partial 
\zeta(t, z, u, v)}{\partial z}.$$

Substituting all these in Equation (\ref{Eq-xi}), we obtain the partial differential equation (PDE) satisfied by 
$\zeta$ as
\begin{eqnarray}
\label{Eq-zeta}
&&
\frac{\partial \zeta(t, z, u, v)}{\partial t} + v s \mu^* z \frac{\partial \zeta(t, z, u, v)}{\partial z}   
+ \left( u r + z \beta(u) - 1 \right) \frac{\partial \zeta(t, z, u, v)}{\partial u} 
+ \left(v s - 1 + \frac{1}{z} \gamma(v) \right)  \frac{\partial \zeta(t, z, u, v)}{\partial v} 
\nonumber\\
&&
\hspace{2 in}
=  - u r \lambda^*  \zeta(t, z, u, v). 
\end{eqnarray}

As usual, applying the method of the characteristics, we reduce the PDE to a system of ODEs, along which 
the solutions are integrated from some initial data given on a suitable hypersurface.

To this end, let $z$, $u$ and $v$ be parameterized by $w$, $0 < w < t$, and 
with the boundary conditions $z(t) = z$, $u(t) = u$, $v(t) = v$.
A comparison with Equation (\ref{Eq-zeta}), taking into account the chain rule, 
gives us
\begin{eqnarray}
\label{chainrule}
\frac{d \zeta \left(t, z(w), u(w), v(w) \right) }{dw} 
&=& \frac{\partial \zeta}{\partial t} \frac{dt}{dw} + \frac{\partial \zeta}{\partial z} \frac{dz}{dw} 
+  \frac{\partial \zeta}{\partial u} \frac{du}{dw} +  \frac{\partial \zeta}{\partial v} \frac{dv}{dw}
\\
&=& - u(w) r \lambda^*  \zeta\left(t, z(w), u(w), v(w)\right).
\nonumber
\end{eqnarray}
This in turn implies that, for a real constant $c$, we are able to deduce that
$$\zeta\left(t, z(w), u(w), v(w)\right) = c \exp{\left\{ - r \lambda^* \int_0^w u(x) \ dx \right\}},$$
where
$c = \zeta\left(0, z(0), u(0), v(0)\right) = \zeta\left(0, z, u, v \right) 
= e^{- u(0) \lambda_0}$, 
and 
$$\zeta\left(t, z(t), u(t), v(t)\right) = c \exp{\left\{ - r \lambda^* \int_0^t u(x) \ dx \right\}},$$
which implies Equation (\ref{EquationZeta}). 

Furthermore, by Equation (\ref{Eq-zeta}) and the chain rule 
written in Equation (\ref{chainrule}), we deduce that 
$$\frac{d z(w)}{dw} = s \mu^* v(w) z(w).$$
Recalling that $0 < w < t$, and that $z(t) = z$,
$$z(w) = C \exp{ \left\{s \mu^* \int_0^w v(x) \ dx \right\}},$$
for a real constant 
$C = z \exp{ \left\{- s \mu^* \int_0^t v(x) \ dx \right\}}$, 
which allows us to write that
$$z(w) 
= z \exp{ \left\{  - s \mu^* \int_w^t v(x) \ dx \right\}}.$$

Similarly, by Equation (\ref{Eq-zeta}) and the chain rule 
written in Equation (\ref{chainrule}), 
and substituting $z(w)$, we get that 
\begin{eqnarray}
\frac{d u(w)}{dw} &=& r u(w) - 1  + z(w) \beta(u(w)) 
\nonumber\\
&=& r u(w)  - 1 + z \beta(u(w)) \exp{ \left\{  - s \mu^* \int_w^t v(x) \ dx \right\}},
\nonumber
\end{eqnarray}
and 
\begin{eqnarray}
\frac{d v(w)}{dw} &=& s v(w) - 1 + \frac{1}{z(w)} \gamma(v(w)) 
\nonumber\\
&=& s v(w) - 1 + \frac{1}{z} \gamma(v(w)) \exp{ \left\{  s \mu^* \int_w^t v(x) \ dx \right\}}.
\nonumber
\end{eqnarray}
By substituting  
$t$ for $t-w$, and by taking into account that we have a change of sign in the LHS of both $\frac{d u(w)}{dw}$ 
and $\frac{d v(w)}{dw}$, we have Equation (\ref{systemuv}), 
with boundary conditions $u(0) = u$ and $v(0) = v$. \qed
\bigskip

\begin{proposition}
Let $v(t) = w'(t) = y_2(t)$ and $w(t) = y_1(t)$, then the system of ODEs as given in Equation (\ref{systemuv}) turns out to be 
a dynamical system such that
\begin{eqnarray}
\label{equy}
u'(t) &=& - r u(t) + 1 -  z \beta \left(u(t)  \right) e^{   - s \mu^* y_1(t)}, 
\\
\label{eq1y}
y'_1(t) &=& y_2(t),
\\
\label{eq2y}
y'_2(t) &=& 1 - s y_2(t) - \frac{\gamma \left(y_2(t)\right)}{z} e^{ s \mu^* y_1(t) }.
\end{eqnarray}
\end{proposition}
\bigskip

\proof
In Theorem \ref{mainTheorem}, for the $Hawkes/sdHawkes/\infty$ system, 
we derive $\zeta(t, z, u, v)$ by Equation (\ref{EquationZeta}) where the couple 
$\left(u(\cdot), v(\cdot) \right)$ solves the system of ODEs as given in Equation (\ref{systemuv}). 
Now, let 
$$\int_0^t v(x) dx = w(t).$$
Differentiating both sides, we have
$v(t) = w'(t)$, which implies that $v'(t) = w''(t)$.
Putting these values in Equation (\ref{systemuv}), we get
\begin{eqnarray}
&&\label{eq1}
u'(t) + r u(t) -1 +  z \beta \left(u(t)  \right) e^{   - s \mu^* w(t)} = 0,
\\
&&\label{eq2}
w''(t) + s w'(t) - 1 + \frac{\gamma \left(w'(t)\right)}{z} e^{ s \mu^* w(t) } = 0.
\end{eqnarray}
To convert Equation (\ref{eq2}) into a dynamical system, let
$$y_1(t) = w(t), \qquad y_2(t) = w'(t), \qquad y'_1(t) = w'(t) = y_2(t),  \qquad y'_2(t) = w''(t)$$
which implies Equations (\ref{equy}), (\ref{eq1y})  and (\ref{eq2y}).
\qed
\bigskip

\begin{proposition}
If the random variables $B$ and $C$ follow exponential distributions 
with parameter $a$ and $b$, respectively,  then Equations (\ref{equy}), (\ref{eq1y})  and (\ref{eq2y}) become such that
\begin{eqnarray}
\label{equye}
u'(t) &=& - r u(t) + 1 -   \frac{z a}{a+u} e^{   - s \mu^* y_1(t)} ,
\\
\label{eq1ye}
y'_1(t) &=& y_2(t) ,
\\
\label{eq2ye}
y'_2(t) &=& 1 - s y_2(t) - \frac{b}{z (b + y_2)} e^{ s \mu^* y_1(t) }.
\end{eqnarray}
Furthermore, let
\begin{equation}\label{setting}
u(t)  = y_1(t),  \qquad v(t)  = y_2(t),  \qquad z(t)  = y_3(t) , \qquad \zeta(t)  = y_4(t),
\end{equation}
then we are able to write the dynamical system
\begin{equation}  
\label{dynamical-system}
\left\{ 
\begin{array}l 
\displaystyle y_1'(t) = 1 - r y_1(t) -  \frac{a y_3(t) }{a + y_1(t) } ,
\\
\\
\displaystyle y_2'(t) = 1 -  s y_2(t) - \frac{b}{y_3(t) \left( b + y_2(t) \right)} ,
\\
\\
\displaystyle y_3'(t) = s \mu^* y_2(t) y_3(t) ,
\\
\\
\displaystyle y'_4(t) = (- \lambda^* r y_1(t)) y_4(t).
\end{array}
\right.
\end{equation}

\end{proposition}
\bigskip

\proof
Taking into account the assumptions made on the random variables $B$ and $C$, 
we get that
\begin{equation}\label{BC}
\beta(u(t)) := \mathbb E [e^{- u(t) B}] = \frac{a}{a+u(t)}, \quad \mbox{and} \quad 
\gamma(v(t)) := \mathbb E [e^{- y_2(t) C}] = \frac{b}{b+y_2(t)}.
\end{equation}
Substituting these values in Equations (\ref{equy}), (\ref{eq1y}) and  (\ref{eq2y}), we get Equations (\ref{equye}), (\ref{eq1ye})  and  (\ref{eq2ye}).
Moreover, taking
$$z(t) = - z e^{- s \mu^* y_1(t)}$$
the system given by Equations (\ref{equye}), (\ref{eq1ye})  and  (\ref{eq2ye}) turns out to be 
\begin{equation}  
\label{system}
\left\{ 
\begin{array}l 
\displaystyle z'(t) = s \mu^* v(t) z(t),
\\
\\
\displaystyle u'(t) + r u(t) -1 +  \frac{z(t) a}{a + u(t)} = 0,
\\
\\
\displaystyle v'(t) + s v(t) - 1 + \frac{b}{z(t) \left( b + v(t) \right)} = 0,
\\
\\
\displaystyle \zeta' = (- \lambda^* r u(t)) \zeta,
\end{array}
\right. ,
\nonumber
\end{equation}
with boundary conditions 
\begin{equation}
\label{boundaryconds}
z(0) = z, \qquad \qquad u(0) = u, \qquad \mbox{and} \qquad v(0) = v.
\end{equation}
Finally, by Equation (\ref{setting}), we are able to write the dynamical system as given in Equation (\ref{dynamical-system}).
\qed
\bigskip

Now, we solve the dynamical system given in Equation (\ref{dynamical-system}) numerically with the boundary conditions given in Equation (\ref{boundaryconds}), 
by the discretization steps along with numerical integrals in MATLAB software.  
For the sake of numerical illustration, the parameter values are chosen such as 
$$\lambda_0 = 2, \qquad \qquad \lambda^* = 2, \qquad \qquad r = 2,  \qquad \qquad s = 2, \qquad \qquad 
\mu^* = 2, \qquad \mbox{and} \qquad  a = b =2,$$
and 
Figure 1 shows the graphs of $\zeta(t,u,v,z)$ versus time.
\begin{figure}[ht!]
\begin{center}
\subfigure{
\includegraphics[scale=0.8]{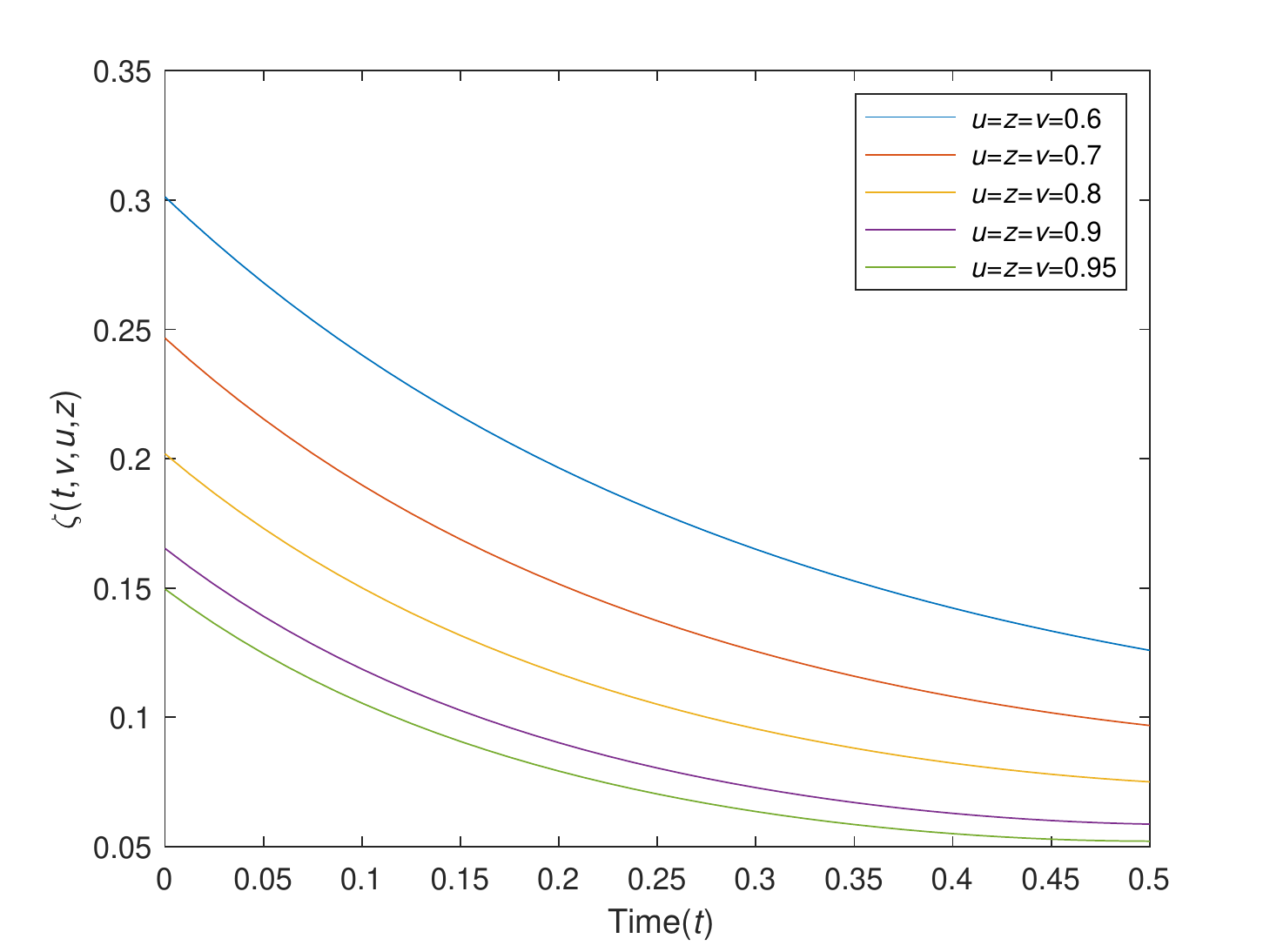}
}
\caption{$Hawkes/sdHawkes/ \infty$ system: Different graphs of $\zeta(t,u,v,z)$ for different values of $u$, $v$ and $z$.}
\end{center}
\end{figure}

Next, we are going to deduce the results for the $M/sdHawkes/\infty$
system.
\bigskip

\begin{theorem}\label{mainTheorem2}
For an $M/sdHawkes/\infty$ system, let
\begin{equation}
\label{Def-zeta-f}
\zeta(t, z, v) = \mathbb E \left[ z^{N_t} e^{-v \mu_t} \right]
\end{equation}
where $t \geq 0$, $0 \leq z \leq 1$, $v \geq 0$.
Given the constant intensity $\lambda_t = \lambda$, 
and the initial values, $(N_t, \mu_t)|_{t=0} = (0, 0)$,  we get that 
\begin{equation}
\label{zeta-f}
\zeta\left(t, z(t), v(t)\right) = \exp{ \left\{- \lambda \int_0^t \left(z(x) + 1\right)  \ dx \right\}},
\end{equation}
where
$$z(w) = z \exp{ \left\{  - s \mu^* \int_w^t v(x) \ dx \right\}},$$
and, given the boundary condition  $v(0) = v$,  $v(\cdot)$ solves the  ODE
\begin{equation}  
\label{ODEv-f}
v'(t) + s v(t) - 1 + \frac{\gamma \left(v(t)\right)}{z} e^{ s \mu^* \int_0^t v(x) \ dx } = 0.
\end{equation}
\end{theorem}
\bigskip

The proof of  Theorem \ref{mainTheorem2} is obtained in three steps: Proposition \ref{firstpartT2} and  Proposition \ref{secondpartT2} given below, whose proofs are in Appendix, and then the main part which is also proved in Appendix. 
\bigskip

\begin{proposition}\label{firstpartT2}
The joint distribution of $\left(N, \mu \right)$, for $t > 0$, and 
$\lambda_t = \lambda$,  positive constant, 
$$F(t, k, \lambda, \mu) = F(t, k, \mu)  \mathbb{I}_{\lambda_t = \lambda}  
= P\left[N_t = k, \mu_t \leq \mu \right]  \mathbb{I}_{\lambda_t = \lambda}$$
is such that

\begin{eqnarray}
\label{LHSF-derivated=RHSF-derivated}
&&
\frac{\partial^2 F(t, k, \mu)}{\partial \mu \partial t} 
+ \frac{\partial}{\partial \mu} \left[ s \left( k \mu^* - \mu \right) 
\frac{\partial F(t, k, \mu)}{\partial \mu} \right]
\\
& = &
\lambda \frac{\partial F(t, k - 1, \mu)}{\partial \mu} 
- (\lambda + \mu) \frac{\partial F(t, k, \mu)}{\partial \mu} 
+  \int_0^\mu y \frac{\partial F(t, k+1, y)}{\partial \mu} dP[C \leq \mu - y]. 
\nonumber
\end{eqnarray}

\end{proposition}
\bigskip

\begin{proposition}\label{secondpartT2}
Taking
$$\xi(t, k, v) = \int_0^\infty  e^{- v \mu} \ \frac{\partial F(t, k, \mu)}{\partial \mu} \ d\mu,$$ 
we have that
\begin{equation}
\label{Eq-xi-f}
\frac{\partial \xi(t, k,  v)}{\partial t} + (v s - 1)  \frac{\partial \xi(t, k,  v)}{\partial v} 
+ \gamma(v) \frac{\partial \xi(t, k+1, v)}{\partial v} 
=  
- \lambda \xi(t, k-1,  v) - (\lambda + v s k \mu^*)  \xi(t, k, v).
\end{equation}

\end{proposition}
\bigskip

\begin{remark}
Note that, in queueing systems, the situation in which a service time follows a  general distribution is more general than the situation in which a service time follows a state-dependent Hawkes process. 
This observation suggests to study the results obtained in Theorem \ref{mainTheorem2} for the $M/sdHawkes/\infty$ system to try to find the relations 
with the analogous results for an $M/G/\infty$ system (see Eick et al.
\cite{EMW}). 
This is an interesting open research problem 
which could be studied in detail.
\end{remark}
\bigskip

For the $M/sdHawkes/\infty$ system, Figure 2 shows the numerical illustration of $\zeta (t,v,z)$ versus time with the parameter values
$$s = 2, \qquad \qquad \lambda = 2, \qquad \qquad b = 2,  \qquad \mbox{and} \qquad \mu^* = 2.$$

\begin{figure}[ht!]
\begin{center}
\subfigure{
\includegraphics[scale=0.8]{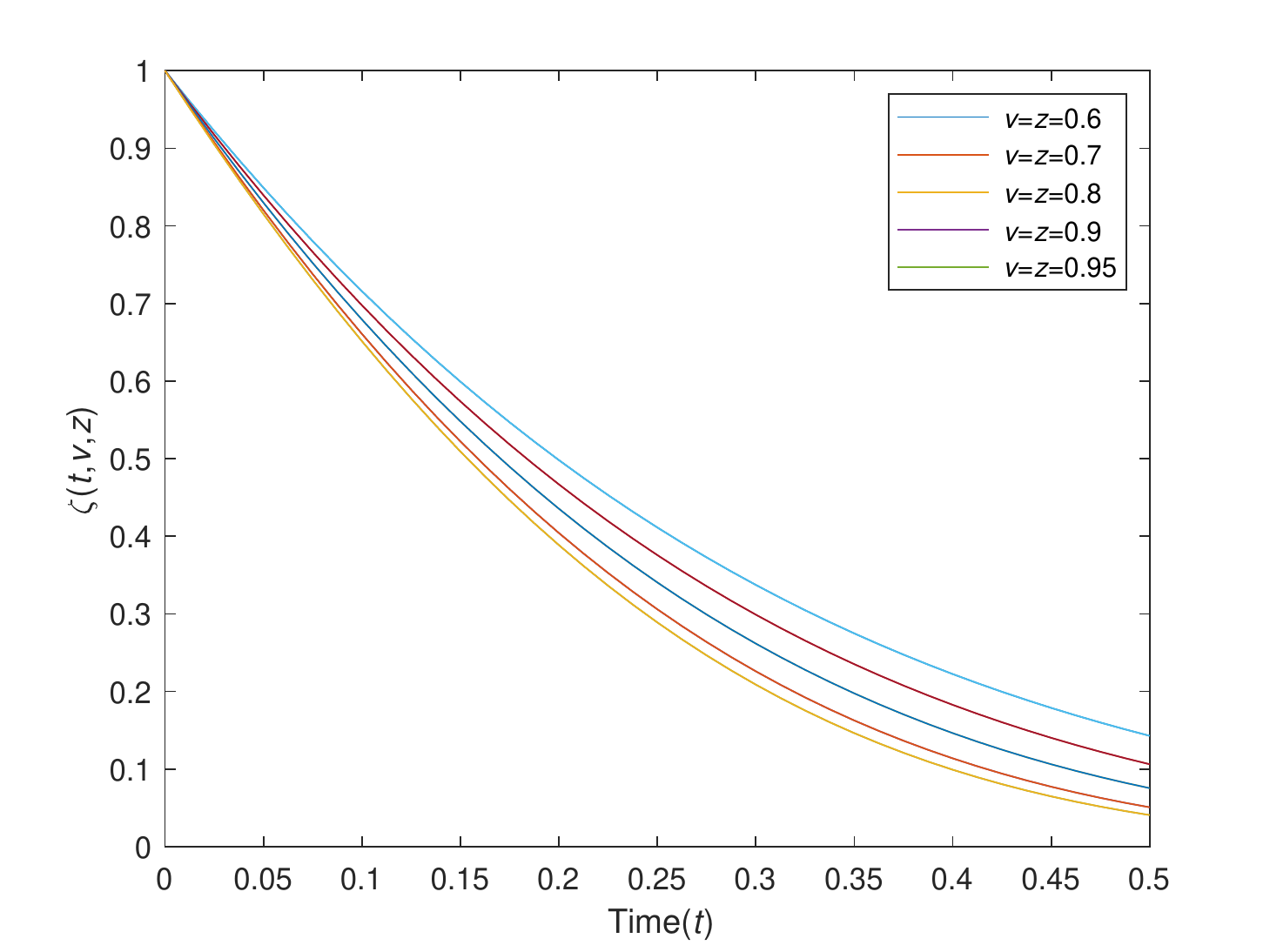}}
\caption{$M/sdHawkes/ \infty$ system: Different graphs of $\zeta(t,v,z)$ for different values of $v$ and $z$.}
\end{center}
\end{figure}


\begin{theorem} \label{corollaryT1}
By the results obtained in Koops et al. 2018, \cite{KSBM}, for an $Hawkes/M/\infty$ system, 
$$\zeta(t, z, u) = \mathbb E \left[ z^{N_t} e^{-u \lambda_t} \right] 
= e^{-u(t) \lambda_0} e^{-\lambda_0 r \int_0^t u(w) \ dw},$$
where $u(\cdot)$ solves the ODE
$$u'(w) = - r u(w) - (1 + (z-1) e^{-\mu^* w}) \beta(u(w)) + 1.$$
\end{theorem}
In Theorem \ref{mainTheorem}, these known results have been extended taking into account the state-dependent Hawkes service times.


For the $Hawkes/M/ \infty$ system, Figure 3 shows the numerical illustration of $\zeta (t,u,z)$ versus time with the parameter values
$$r = 2, \qquad \qquad \lambda_0 = 2, \qquad \qquad a = 2,  \qquad \mbox{and} \qquad \mu^* = 2.$$

\begin{figure}[ht!]
\begin{center}
\subfigure{
\includegraphics[scale=0.8]{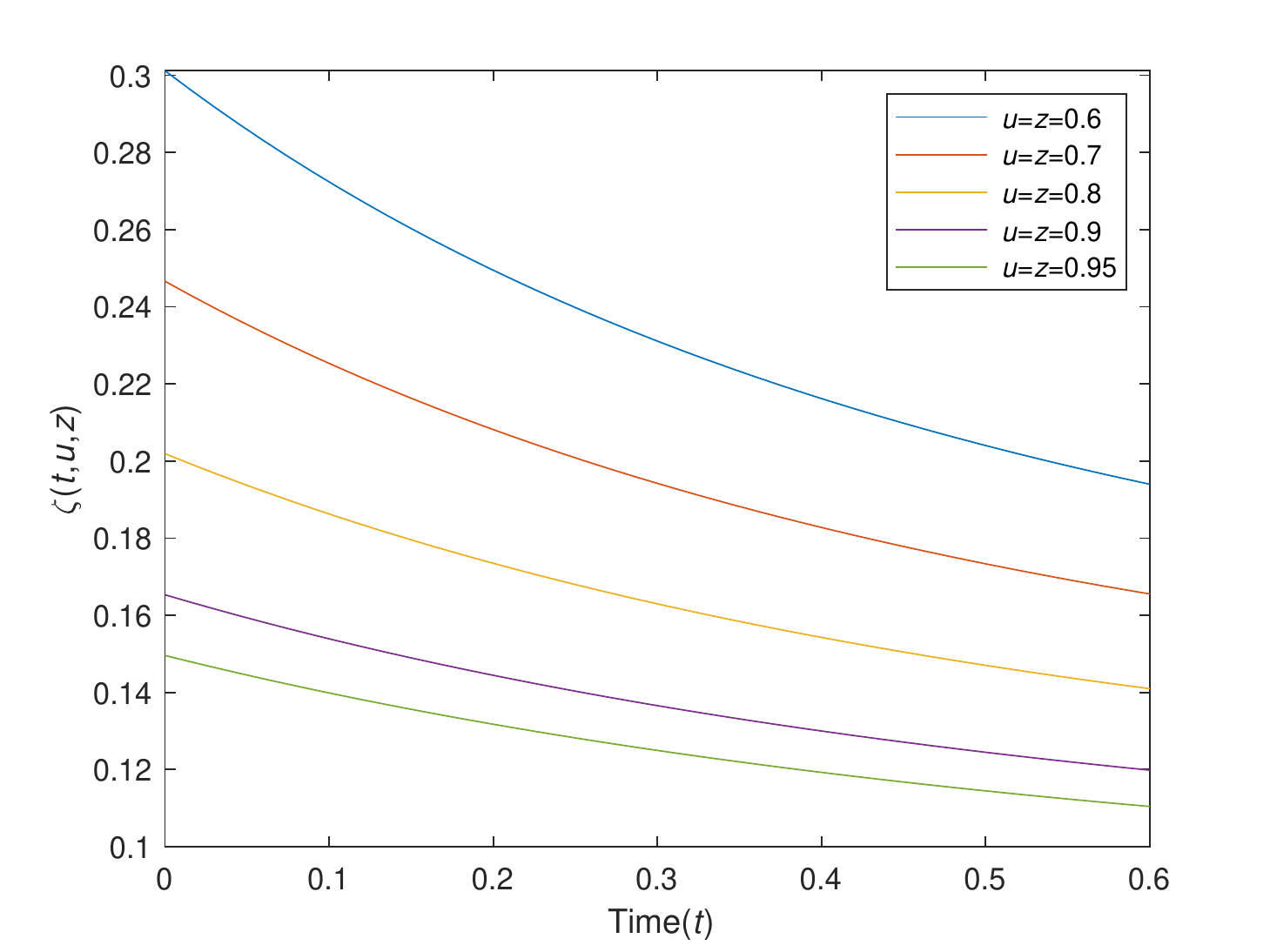}
}
\caption{$Hawkes/M/ \infty$ system: Different graphs of $\zeta(t,u,z)$ for different values of $u$ and $z$.}
\end{center}
\end{figure}

\bigskip

\begin{theorem} \label{Theorem4}
Under all the assumptions of Theorem \ref{mainTheorem}, 
if $r = s = 0$ and  $\mathbb E[B_t] = \mathbb E[C_t] = 0$, 
then the join distribution of $\left(N, \lambda, \mu \right)$, for $t > 0$ is 
$$F(t, k, \lambda, \mu) 
= F(t, k) \mathbb{I}_{\lambda_t = \lambda_0} \mathbb{I}_{\mu_t = N_t \mu_0}$$
where, for $k=0, 1, 2, \ldots$, 
\begin{eqnarray}
\label{F=P=}
F(t, k) =  \frac{1}{k!} \left(\frac{\lambda_0}{\mu_0} (1 - e^{-\mu_0 t})\right)^k 
\exp{\left\{- \frac{\lambda_0}{\mu_0} (1 - e^{-\mu_0 t}) \right\} },
\end{eqnarray}
and
\begin{eqnarray}
\label{zetaF=P}
\zeta(t, z, u, v) = e^{-u \lambda_0} \exp\left\{ - \frac{\lambda_0}{\mu_0} (1 - e^{-\mu_0 t}) (1 - z e^{-v \mu_0}) ) \right\}. 
\end{eqnarray}

\end{theorem}
\begin{proof}
In Theorem \ref{mainTheorem},  
if $r = s= 0$ and $\mathbb E[B_t] = \mathbb E[C_t] = 0$, 
(Gross et al. 2019 \cite{GSTH}),
Equations (\ref{lambda}) and (\ref{mu}) are simplified and
$\label{lambda+mu-simplified}
\lambda_t = \lambda_0 
$ and $
\mu_t = N_t  \mu_0$. 
This means that the arrival process, 
$M_t$, turns to be a Poisson process with parameter $\lambda_0$ and, if there are $N_t$ customers in 
the system at time $t$, the service time follows an exponential distribution with parameter  $N_t \mu_0$, 
i.e., this is an $M/M/\infty$ system and the join distribution is 
$$F(t, k, \lambda, \mu) = P[N_t = k, \lambda_t = \lambda_0, \mu_t = N_t \mu_0] 
= P[N_t = k] \mathbb{I}_{\lambda_t = \lambda_0} \mathbb{I}_{\mu_t = N_t \mu_0}.$$
Hence, we have that 
$F(t, k, \lambda, \mu) = F(t, k) = P[N_t = k]$ and 
$$\frac{dF(t, 0)}{dt} = \lim_{\Delta t \rightarrow 0} \frac{F(t + \Delta t, 0) - F(t, 0)}{\Delta t} 
= - \lambda_0 F(t, 0) + \mu_0 F(t, 1).$$
For $k = 1, 2, \ldots$
\begin{eqnarray}
\frac{dF(t, k)}{dt}  
&=& \lambda_0 F(t, k-1) - (\lambda_0 + k \mu_0) F(t, k) + (k+1) \mu_0 F(t, k+1).
\nonumber
\end{eqnarray}
Solving the above system of difference-differential equations, 
we obtain Equation (\ref{F=P=}).
Consequently, to get Equation (\ref{zetaF=P}), as usual, note that
$$\zeta(t, z, u, v) = e^{-u \lambda_0} \mathbb E\left[ (z e^{-v \mu_0})^{N_t} \right]
= e^{-u \lambda_0} \sum_{k=0}^\infty  (z e^{-v \mu_0})^k P[N_t = k]. \qed$$
\end{proof}
\bigskip

\begin{remark}
\label{MMinfty}
In Theorem \ref{Theorem4}, $F(t,k)$, for $t \rightarrow \infty$, turns out to be
$$\displaystyle P[N_\infty = k] = \frac{1}{k!} \left( \frac{\lambda_0}{\mu_0} \right)^k e^{- \lambda_0 / \mu_0},~~ 
\mbox{for}~~ k = 0, 1, 2, \ldots. $$ 
These results of distribution of $N_t$ 
and of limiting distribution coincide with the results presented in 
Gross et al. 
\cite{GSTH}.
\end{remark}
\bigskip

\section{Concluding Remarks and Open Problems}

Main contribution of this paper is to derive the joint time-dependent distribution of the vector process $(N, \lambda, \mu)$, given by system size, arrival and service intensity processes for the $Hawkes/sdHawkes/\infty$   and 
$M/sdHawkes/\infty$ systems. 
To get this task,  we proved the Markov property for the vector process $(N, \lambda, \mu)$. Then, 
the idea is to characterize the law of the infinite server system in terms of the solution of a corresponding system of ODEs.
The methodology used is inspired by Koops et al.
\cite{KSBM}, where the authors  achieved the same results for 
a $Hawkes/M/\infty$ system. 
For a $M/M/\infty$ system, 
the above results are deduced directly following a classical procedure. 

However, we have to take into account that, in a queueing system, if the service time follows a general distribution, then this is a more general situation than that in which the service time follows a state-dependent Hawkes process. 
Hence, an interesting research problem, which could be studied in detail, is to connect the results obtained 
for the 
$M/sdHawkes/\infty$ system with the analogous results for the $M/G/\infty$ system. More precisely, the idea is to investigate the relations between the transient behaviour of 
$M/G/\infty$ system and that of  $M/sdHawkes/\infty$ system.

Taking into account that Koops et al. 
\cite{KSBM} obtained the time-dependent results for $Hawkes/M/\infty$ system, an open problem is to study if these results could be extended for a $Hawkes/G/\infty$ system. 
Instead of 
obtaining a solution by solving ODEs (as in Theorem \ref{mainTheorem}), 
several methods are possible, for instance, 
methods such as fixed point equation in the transform domain and concepts using branching processes. 
Other future work could focus on the 
asymptotic behaviour of the distributions and the moments for the $Hawkes/sdHawkes/\infty$ system. 
\bigskip

\section{Appendix}
\subsection{Proof of Proposition \ref{firstpart}}
%
%
Taking into account the observations made in Remark \ref{Nlambdamu markov}, and the results of Lemma \ref{lambdamut+dt}, since in $(t, t + \Delta t)$, we can have a service completed with probability $y  \Delta t$, or we can have still customers  in service with 
probability $1 - y  \Delta t$, 
\begin{eqnarray}
F\left(t + \Delta t, k, \lambda - r \Delta t \left( \lambda -  \lambda^* \right), \mu - s \Delta t \left( \mu -  k \mu^* \right)\right) 
&  =& 
 \int_0^\lambda \frac{\partial F(t, k-1, x, \mu)}{\partial \lambda} P[B \leq \lambda - x] \ x  \Delta t \ dx
\nonumber\\
&&  + \int_0^\mu \frac{\partial F(t, k+1, \lambda, y)}{\partial \mu} P[C \leq \mu - y] \ y  \Delta t \ dy
\nonumber\\
&&   + \int_0^\mu \frac{\partial F(t, k, \lambda, y)}{\partial \mu} \left(1 - y  \Delta t \right) \ dy
+ \int_0^\lambda \frac{\partial F(t, k, x, \mu)}{\partial \lambda}   \left(- x  \Delta t \right) \ dx.
\nonumber
\end{eqnarray}
After elementary manipulations, we get 
\begin{eqnarray}
&&\frac{1}{\Delta t} \left[F(t + \Delta t, k, \lambda - r \Delta t \left( \lambda -  \lambda^* \right), \mu - s \Delta t \left( \mu -  k \mu^* \right)) 
- F(t, k, \lambda, \mu) \right]
\nonumber\\
&&\hspace{1 in} =
\int_0^\lambda \frac{\partial F(t, k-1, x, \mu)}{\partial \lambda} P[B \leq \lambda - x] \ x  \ dx
- \int_0^\lambda \frac{\partial F(t, k, x, \mu)}{\partial \lambda}   x  \ dx 
\nonumber\\
&&\hspace{1.2 in}   + \int_0^\mu \frac{\partial F(t, k+1, \lambda, y)}{\partial \mu} P[C \leq \mu - y] \ y  \ dy
- \int_0^\mu \frac{\partial F(t, k, \lambda, y)}{\partial \mu} y \ dy.
\nonumber
\end{eqnarray}
Then, letting $\Delta t \downarrow 0$, we have
\begin{eqnarray}
\label{Fderivated}
&&\frac{\partial F(t, k, \lambda, \mu)}{\partial t} 
- r \left( \lambda -  \lambda^* \right) \frac{\partial F(t, k, \lambda, \mu)}{\partial \lambda} 
- s \left( \mu -  k \mu^*  \right) \frac{\partial F(t, k, \lambda, \mu)}{\partial \mu} 
\nonumber\\
&& \hspace{1 in} = 
\int_0^\lambda x P[B \leq \lambda - x] \frac{\partial F(t, k-1, x, \mu)}{\partial \lambda}   \ dx
- \int_0^\lambda x \frac{\partial F(t, k, x, \mu)}{\partial \lambda}  \ dx 
\\
&& \hspace{1.2 in} +  \int_0^\mu y P[C \leq \mu - y] \frac{\partial F(t, k+1, \lambda, y)}{\partial \mu} \ dy
- \int_0^\mu y \frac{\partial F(t, k, \lambda, y)}{\partial \mu}  \ dy.
\nonumber
\end{eqnarray}
Taking the left hand side, LHS, of Equation (\ref{Fderivated}) and differentiating, successively, 
with respect to $\lambda$ and $\mu$, we get
\begin{eqnarray}
\label{LHSFderivated}
\frac{\partial^2 LHS}{\partial \mu \partial \lambda} 
=
\frac{\partial^3 F(t, k, \lambda, \mu)}{\partial \mu \partial \lambda \partial t} 
+  \frac{\partial}{\partial \lambda} 
\left[ r \left( \lambda^* - \lambda \right) 
\frac{\partial^2 F(t, k, \lambda, \mu)}{\partial \mu \partial \lambda} \right]
+ \frac{\partial}{\partial \mu} 
\left[ s \left( k \mu^* - \mu \right) 
\frac{\partial^2 F(t, k, \lambda, \mu)}{\partial \mu \partial \lambda} \right].
\end{eqnarray}
Recalling that $P[B \leq 0] = 0$, $P[C \leq 0] = 0$, then 
taking the right hand side, RHS, of Equation (\ref{Fderivated}), and, 
again, differentiating, successively, with respect to $\lambda$ and to $\mu$, we get
\begin{eqnarray}
\label{RHSFderivated}
&&\frac{\partial^2 RHS}{\partial \mu \partial \lambda}  
=
\int_0^\lambda x  \frac{\partial^2 F(t, k-1, x, \mu)}{\partial \mu \partial \lambda}   \ dP[B \leq \lambda - x]
\\
&&\qquad \qquad 
- (\lambda + \mu) \frac{\partial^2 F(t, k, \lambda, \mu)}{\partial \mu \partial \lambda} 
+  \int_0^\mu y \frac{\partial^2 F(t, k+1, \lambda, y)}{\partial \lambda \partial \mu} \ dP[C \leq \mu - y]. \qed
\nonumber
\end{eqnarray}
\bigskip

\subsection{Proof of Proposition \ref{secondpart}}

In order to transform Equation (\ref{LHSFderivated=RHSFderivated}) 
with respect to the intensities $\lambda$ and $\mu$, note that
$$\int_0^\infty  \int_0^\infty e^{- u \lambda} e^{- v \mu} \ \frac{\partial^3 F(t, k, \lambda, \mu)}{\partial \mu \partial \lambda \partial t}  \ d\lambda \ d\mu
= \frac{\partial \xi(t, k, u, v)}{\partial t},$$
\begin{eqnarray}
&&\int_0^\infty  \int_0^\infty e^{- u \lambda} e^{- v \mu} \ 
\frac{\partial}{\partial \lambda} \left[ r \left( \lambda^* - \lambda \right) \frac{\partial^2 F(t, k, \lambda, \mu)}{\partial \mu \partial \lambda} \right] \ d\lambda \ d\mu 
\nonumber\\
&& \qquad \qquad  \qquad = \int_0^\infty  
\int_0^\infty u r  e^{- u \lambda} e^{- v \mu} \ \left( \lambda^* - \lambda \right) \frac{\partial^2 F(t, k, \lambda, \mu)}{\partial \mu \partial \lambda} \ d\lambda \ d\mu
\nonumber\\
&& \qquad \qquad  \qquad = 
u r \lambda^* \xi(t, k, u, v) + u r  \frac{\partial \xi(t, k, u, v)}{\partial u},
\nonumber
\end{eqnarray}
and, analogously,
$$\int_0^\infty  \int_0^\infty e^{- u \lambda} e^{- v \mu} \ 
\frac{\partial}{\partial \mu} \left[ s \left( k \mu^*   - \mu \right) \frac{\partial^2 F(t, k, \lambda, \mu)}{\partial \mu \partial \lambda} \right] \ d\lambda \ d\mu 
= v s k \mu^*  \xi(t, k, u, v) + v s  \frac{\partial \xi(t, k, u, v)}{\partial v}.$$
Moreover,  recalling that 
$\beta(u) := \mathbb E [e^{- u B}]$ and $\gamma(v) := \mathbb E [e^{- v C}]$, 
successively we have that
\begin{eqnarray}
&&\int_0^\infty  \int_0^\infty e^{- u \lambda} e^{- v \mu} \  
\int_0^\lambda x  \frac{\partial^2 F(t, k-1, x, \mu)}{\partial \mu \partial \lambda}   \ dP[B \leq \lambda - x] \ d\lambda \ d\mu 
\nonumber\\
&& \qquad = - \int_0^\infty \int_0^\lambda e^{- u (\lambda - x)} 
\frac{\partial^2 
\xi(t, k-1, u, v)}{\partial \lambda \partial u}  \ dP[B \leq \lambda - x] \ d\lambda 
= - \beta(u) \frac{\partial \xi(t, k-1, u, v) }{\partial u},
\nonumber
\end{eqnarray}
\begin{eqnarray}
&&
\int_0^\infty  \int_0^\infty e^{- u \lambda} e^{- v \mu} \  
\int_0^\mu y \frac{\partial^2 F(t, k+1, \lambda, y)}{\partial \lambda \partial \mu} \ dP[C \leq \mu - y]\ d\lambda \ d\mu 
= - \gamma(v) \frac{\partial \xi(t, k+1, u, v)}{\partial v}, 
\nonumber
\end{eqnarray}
and 
$$- \int_0^\infty  \int_0^\infty e^{- u \lambda} e^{- v \mu} \  (\lambda + \mu) \frac{\partial^2 F(t, k, \lambda, \mu)}{\partial \mu \partial \lambda} \ d\lambda \ d\mu
=  \frac{\partial \xi(t, k, u, v)}{\partial u} + \frac{\partial \xi(t, k, u, v)}{\partial v}.$$
After all the transformations we have made and rearranging, we get the claim. \qed
\bigskip

\subsection{Proof of Proposition \ref{firstpartT2}}

Recalling the results achieved in Lemma \ref{lambdamut+dt}, since in a small time interval, $(t, t + \Delta t)$, we can have a service with probability $y  \Delta t$, or we can have still customers  in service with 
probability $1 - y  \Delta t$, 
\begin{eqnarray}
&&\hspace{-20 pt}F\left(t + \Delta t, k, \mu - s \Delta t \left( \mu -  k \mu^* \right)\right) 
\nonumber\\
&&= \lambda \Delta t F(t, k-1, \mu) -  \lambda \Delta t F(t, k, \mu)  
+ \int_0^\mu \frac{\partial F(t, k+1, y)}{\partial \mu} P[C \leq \mu - y] \ y  \Delta t \ dy
+ \int_0^\mu \frac{\partial F(t, k,  y)}{\partial \mu} \left(1 - y  \Delta t \right) \ dy
\nonumber
\end{eqnarray}
After elementary manipulations, we get 
\begin{eqnarray}
&&\hspace{-20 pt}\frac{1}{\Delta t} \left[F(t + \Delta t, k, \mu - s \Delta t \left( \mu -  k \mu^* \right)) 
- F(t, k, \mu) \right]
\nonumber\\
&&=
 \lambda F(t, k-1, \mu) -  \lambda F(t, k, \mu) 
 +  \int_0^\mu \frac{\partial F(t, k+1, y)}{\partial \mu} P[C \leq \mu - y] \ y \ dy
- \int_0^\mu \frac{\partial F(t, k,  y)}{\partial \mu}  y   \ dy.
\nonumber
\end{eqnarray}
Then, letting $\Delta t \downarrow 0$, we have
\begin{eqnarray}
\label{F-derivated}
&&\hspace{-20 pt}\frac{\partial F(t, k, \mu)}{\partial t} 
- s \left( \mu -  k \mu^*  \right) \frac{\partial F(t, k, \lambda, \mu)}{\partial \mu} 
\nonumber\\
&&= \lambda F(t, k-1, \mu)-\lambda F(t, k, \mu) 
+  \int_0^\mu \frac{\partial F(t, k+1, y)}{\partial \mu} P[C \leq \mu - y] \ y \ dy 
 - \int_0^\mu \frac{\partial F(t, k,  y)}{\partial \mu}  y   \ dy.
\end{eqnarray}
Taking  LHS of Equation (\ref{F-derivated}) and differentiating 
with respect to $\mu$ yields
\begin{eqnarray}
\label{LHSF-derivated}
\frac{\partial LHS}{\partial \mu} 
&=&
\frac{\partial^2 F(t, k, \mu)}{\partial \mu \partial t} 
+ \frac{\partial}{\partial \mu} \left[ s \left( k \mu^* - \mu \right) 
\frac{\partial F(t, k, \mu)}{\partial \mu} \right].
\end{eqnarray}
Recalling that $P[C \leq 0] = 0$, then 
taking RHS of Equation (\ref{F-derivated}), and, 
again, differentiating with respect to  $\mu$, we get
\begin{eqnarray}
\label{RHSF-derivated}
\frac{\partial RHS}{\partial \mu}  
& = &
\lambda \frac{\partial F(t, k - 1, \mu)}{\partial \mu} 
- (\lambda + \mu) \frac{\partial F(t, k, \mu)}{\partial \mu} 
+  \int_0^\mu y \frac{\partial F(t, k+1, y)}{\partial \mu} dP[C \leq \mu - y]. 
\end{eqnarray}
Rearranging Equation (\ref{LHSF-derivated}) and Equation (\ref{RHSF-derivated}),  we get the claim. \qed
\bigskip

\subsection{Proof of Proposition \ref{secondpartT2}}

Note that
$$\int_0^\infty  e^{- v \mu} \ \frac{\partial^2 F(t, k, \mu)}{\partial \mu \partial t}  \ d\mu
= \frac{\partial \xi(t, k, v)}{\partial t},$$
\begin{eqnarray}
&&\int_0^\infty  e^{- v \mu} \  \frac{\partial}{\partial \mu} \left[ s \left( k \mu^*   - \mu \right) \frac{\partial F(t, k, \mu)}{\partial \mu} \right]  \ d\mu
= v s k \mu^*  \xi(t, k, v) + v s  \frac{\partial \xi(t, k, v)}{\partial v}.
\nonumber
\end{eqnarray}
Therefore, by  the LHS of Equation (\ref{LHSF-derivated=RHSF-derivated}), 
\begin{equation}
\label{LHS-transformed}
\frac{\partial \xi(t, k,  v)}{\partial t} + v s  \frac{\partial \xi(t, k,  v)}{\partial v}
+  v s k \mu^* \xi(t, k, v).
\end{equation}
Recalling that 
$\gamma(v) := \mathbb E [e^{- v C}]$, 
by the RHS of Equation (\ref{LHSF-derivated=RHSF-derivated}), we have
\begin{eqnarray}
&&\int_0^\infty e^{- v \mu} \  
\int_0^\mu y \frac{\partial F(t, k+1, y)}{\partial \mu} \ dP[C \leq \mu - y] \ d\mu 
= - \gamma(v) \frac{\partial \xi(t, k+1, v)}{\partial v}, 
\nonumber
\end{eqnarray}
and 
$$- \int_0^\infty  e^{- v \mu} \  (\lambda + \mu) \frac{\partial F(t, k, \mu)}{\partial \mu} \ d\mu
= - \lambda  \xi(t, k, v) + \frac{\partial \xi(t, k, v)}{\partial v}.$$
After all the transformations we made on Equation (\ref{LHSF-derivated=RHSF-derivated}),  and by Equation (\ref{LHS-transformed}), we get the claim. \qed
\bigskip

\subsection{Proof of Theorem \ref{mainTheorem2}}

Recalling the definition of $\zeta$ given in Equation (\ref{Def-zeta-f}), which in turn implies that
$$\zeta(t, z, v) = \sum_{k=0}^\infty z^k \xi(t, k, v) = \mathbb E \left[ z^{N_t} e^{-v \mu_t} \right],$$
note that
$$\sum_{k=0}^\infty z^k  \frac{\partial \xi(t, k, v)}{\partial t}  
= \frac{\partial \zeta(t, z, v)}{\partial t},$$
$$\sum_{k=0}^\infty z^k  (v s - 1)  \frac{\partial \xi(t, k,  v)}{\partial v} 
= (v s - 1) \frac{\partial \zeta(t, z, v)}{\partial v},$$
and 
\begin{eqnarray}
\sum_{k=0}^\infty z^k  \gamma(v)  \frac{\partial \xi(t, k+1, v)}{\partial v} 
&=&  \gamma(v)  \frac{1}{z} \frac{\partial}{\partial v}  \sum_{k=0}^\infty z^{k+1}  \xi(t, k+1, v)
= \gamma(v)  \frac{1}{z} \frac{\partial  \zeta(t, z, v)}{\partial v}.
\nonumber
\end{eqnarray}
Moreover, taking into account that $F(t, -1,  \mu) = 0$
\begin{eqnarray}
\sum_{k=0}^\infty z^k \lambda \xi(t, k-1,  v) 
&=& \lambda \left( \xi(t, -1,  v) + \sum_{k=1}^\infty z^k \xi(t, k-1,  v) \right)
= \lambda  z \zeta(t, z, v),
\nonumber
\end{eqnarray}
and 
\begin{eqnarray}
\sum_{k=0}^\infty z^k (\lambda + v s k \mu^*)  \xi(t, k, v) 
&=& \lambda  \zeta(t, z, v) + v s \mu^* z \frac{\partial  \zeta(t, z, v) }{\partial z}.
\nonumber
\end{eqnarray}
Applying all these transformations on Equation (\ref{Eq-xi-f}), we find for $\zeta(t, z, v)$ the PDE
\begin{eqnarray}\label{Eq-zeta-f}
&&\frac{\partial \zeta(t, z, v)}{\partial t} + v s \mu^* z \frac{\partial  \zeta(t, z, v)}{\partial z} 
+ \left(v s - 1 + \frac{\gamma(v)}{z} \right) \frac{\partial \zeta(t, z, v)}{\partial v} 
\nonumber\\
&& \qquad \qquad = - \lambda  (z+ 1) \zeta(t, z, v).
\nonumber
\end{eqnarray}
Applying the method of the characteristics, as in Theorem \ref{mainTheorem}, we reduce a PDE to a system of ODEs.
Again, let $z$ and $v$ be parameterized by $w$, 
with the boundary conditions $z(t) = z$,  $v(t) = v$ and $0 < w < t$.  
Taking into account that for the chain rule, 
and by a comparison with Equation (\ref{Eq-zeta-f}) we get
$$\frac{d \zeta \left(t, z(w), v(w) \right) }{dw} 
= \frac{\partial \zeta}{\partial t} \frac{dt}{dw} + \frac{\partial \zeta}{\partial z} \frac{dz}{dw} 
+  \frac{\partial \zeta}{\partial v} \frac{dv}{dw} 
= - \lambda (z(w)+1)  \zeta\left(t, z(w),  v(w)\right).$$
Again, following the same steps of Theorem \ref{mainTheorem}, we get the claim. \qed
\bigskip

\section*{Acknowledgments}
The authors are grateful to the editor and the anonymous reviewers for their insightful comments,
which helped in improving the paper. One of the authors (PT) gratefully acknowledges  the support received from the Department of
Mathematics, Indian Institute of Technology Delhi, India. This research work is supported by the Department of Science and Technology, India.

\end{document}